\documentclass[12pt]{article}
\usepackage{amsmath, amsfonts,amsthm,amssymb,amsbsy,upref,color,graphicx,amscd,hyperref,makeidx,enumerate}
\usepackage{tikz}
\usepackage[active]{srcltx}
\usepackage[latin1]{inputenc}
\usepackage{enumerate}
\usepackage[english]{babel}
\usepackage{newlfont}
\usepackage{amsbsy}
\usepackage[english]{babel}
\usepackage[center]{caption2}
\usepackage{amssymb}
\usepackage{amsmath}
\usepackage{latexsym}
\usepackage{amsthm}
\usepackage{amsthm}
\theoremstyle{plain}
\newtheorem{thm}{Theorem}

\newtheorem{nota}[thm]{Notation}

\newtheorem{defin}[thm]{Definition}

\newcommand{\R}{\mathbb{R}}

\newcommand{\N}{\mathbb{N}}
\newcommand{\C}{\mathbb{C}}

\usepackage{mathtools}

\def\multiset#1#2{\ensuremath{\left(\kern-.2em\left(\genfrac{}{}{0pt}{}{#1}{#2}\right)\kern-.2em\right)}}
\usepackage[center]{caption2}

\begin{document}

\title{Buneman's theorem \\ for trees with exactly $n$ vertices}
\author{Agnese Baldisserri}

\date{}

\maketitle

\def\thefootnote{}

\footnotetext{ \hspace*{-0.36cm}


 }

\begin{abstract}
Let  ${\cal T}=(T,w)$ be a positive-weighted tree with at least $n$ vertices.
For any  $i,j \in \{1,...,n\}$,
let $D_{i,j} ({\cal T})$ be
the weight of the unique path in $T$ connecting $i$ and $j$. The  $D_{i,j} ({\cal T})$ are called 
$2$-weights of  ${\cal T}$ and, if we put in order the $2$-weights, the vector which has the $D_{i,j} ({\cal T})$ as components is called \emph{$2$-dissimilarity vector} of $ {\cal T}$. Given a family 
of positive real numbers $\{D_{i,j}\}_{i,j \in \{1,...,n\}}$,
we say that a positive-weighted tree ${\cal T}=(T,w)$ realizes the family 
if $\{1,...,n\} \subset V(T)$ and $D_{i,j}({\cal T})=D_{i,j}$ for any $ i,j  \in \{1,...,n\}$.

A characterization of $2$-dissimilarity families of positive weighted trees is already known (see \cite{B}, \cite{SimP} or \cite{St}): the families must satisfy the well-known \emph{four-point condition}. However we can wonder when there exists a positive-weighted tree with \emph{exactly} $n$ vertices, $1,...,n,$ and realizing the family $\{D_{i,j}\}$. In this paper we will show that the four-point condition is necessary but no more sufficient, and so we will introduce two additional conditions (see Theorem \ref{thm:ThmAgne}).
\end{abstract}

\section{Introduction}
For any tree $T$, let $E(T)$, $V(T)$ and $L(T)$ 
be respectively the set of the edges,   
the set of the vertices and  the set of the leaves of $T$.
A {\bf weighted tree} ${\cal T}=(T,w)$ is a tree $T$ 
endowed with a function $w: E(T) \rightarrow \R$. 
For any edge $e$, the real number $w(e)$ is called the weight of the edge. If the weight of every edge is positive, we say that the tree is {\bf positive-weighted}; if the weight of every edge is nonnegative, we say that the tree is {\bf nonnegative-weighted}.
For any finite subtree $T'$ of  $T$, we define $w(T')$ to be the sum of the weights of the edges of $T'$. \\
In this paper we will deal only with finite positive-weighted trees.

\begin{defin} \label{defin:weight}
Let ${\cal T}=(T,w) $ be a positive-weighted tree. 
For any distinct $ i,j \in V(T)$,
 we define $ D_{\{i,j\}}({\cal T}) $ to be the weight of the unique path joining $i$ with $j$. 
We call such a subtree ``the subtree realizing  $ D_{\{i,j \}}({\cal T}) $''.
We define $ D_{\{i,i\}}({\cal T})=0 $ for any $i \in V(T)$. More simply, we denote 
$D_{\{i,j\}}({\cal T})$ by
$D_{i,j}({\cal T})$ for any order of $i,j$.  
We call  the  $ D_{i,j}({\cal T})$ 
the {\bf $2$-weights} of ${\cal T}$.
\end{defin}

If $S $ is a subset of $V(T)$, $|S|=n$, and  we order in some way the $2$-subsets of 
$ S$ (for instance, we order $S$ in some way 
and then we order the $2$-subsets of $ S$
in the lexicographic order with respect to the order of $S$),
the $2$-weights with this order give
a vector in $\mathbb{R}^{ n \choose 2}$. This vector is called 
$2${\bf -dissimilarity vector} of $({\cal T}, S)$.
Equivalently, if we don't fix any order, we can speak 
of the {\bf family of the $2$-weights}  of $({\cal T}, S)$.   

We can wonder when a family of positive real numbers is the family of the $2$-weights
of some weighted tree and of some subset of the set of its vertices. 
If $S$ is a finite set, we say that a family of positive real numbers
 $\{D_{i,j}\}_{i,j \in \{1,...,n\}}$  is {\bf p-treelike} (respectively nn-treelike) if there exist a
positive-weighted (respectively  nonnegative-weighted)  tree
 ${\cal T}=(T,w)$ and a subset $S$ 
of the set of its vertices such that $ D_{i,j}({\cal T}) = D_{i,j}$  for any 
$i,j \in \{1,...,n\}$.
If the tree is a positive-weighted (respectively  nonnegative-weighted) tree and $S \subset L(T)$,
we say that the family
 is {\bf p-l-treelike} (respectively nn-l-treelike).
A criterion for a metric on
a finite set to be p-treelike
was established in  \cite{B}, \cite{SimP}, \cite{St}: 

\begin{thm} \label{Bune} %{\bf (Buneman)} 
Let $\{D_{i,j}\}_{i,j \in \{1,...,n\}}$ be a set of positive real numbers. 
It is p-treelike if and only if, for all $i,j,k,h  \in \{1,...,n\}$,
the maximum of $$\{D_{i,j} + D_{k,h},D_{i,k} + D_{j,h},D_{i,h} + D_{k,j}\}$$ is attained at least twice. 
\end{thm}

This condition is called \emph{four-point condition} and it is stronger than the triangle inequalities (just put $h=k$). Moreover, it is easy to prove that this condition is necessary and sufficient to ensure that a family of positive real numbers $\{D_{i,j}\}$ is nn-l-treelike.

Definition \ref{defin:weight} can be extended also to graphs: given two distinct vertices $i,j$ of a positive-weighted graph $\cal{G}=(G,w)$, $D_{i,j}(\cal{G})$ is defined to be the minimum of the weights of the paths joining $i$ with $j$; we define $D_{i,i}(\cal{G})=0$ as before. See for example \cite{H-Y} for some results on $2$-dissimilarity vectors of graphs. In \cite{H-P} and \cite{B-S} the definition of $2$-weights is given also for general-weighted graphs and trees respectively (graphs and trees in which the edges have real weights): in the first the authors give a characterization of the families of real numbers that are the $2$-dissimilarity families of a general-weighted graph, whereas the second deals with the case of general-weighted trees and their $2$-dissimilarity vectors. 

Finally we want to recall that, given a positive-weighted graph $\cal{G}=(G,w)$ with $\{1,...,n\} \subset V(G)$, there exists a definition of $k$-weights, where $k$ is a natural number less than $n$; when $k=2$ this definition corresponds exactly with the definition of $2$-weights (for both positive-weighted graphs and trees):

\begin{defin}
Let ${\cal G}=(G,w) $ be a positive-weighted graph. 
For any distinct $ i_1, .....,i_k \in V(G)$,
 we define $$ D_{\{i_1,...., i_k\}}({\cal G}) 
= min 
\{w(R) | \; R \text{ a connected subgraph of } G  \text{ such that } V(R) \ni 
i_1,...., i_k\}.$$ More simply, we can denote 
$D_{\{i_1,...., i_k\}}({\cal G})$ by
$D_{i_1,...., i_k}({\cal G})$ for any order of $i_1,..., i_k$. The  $ D_{i_1,...., i_k}({\cal G})$ are called
the {\bf $k$-weights} of ${\cal G}$.
\end{defin}
 
See \cite{P-S}, \cite{H-H-M-S}, \cite{B-R2}, \cite{Man}, \cite{Ru1} and \cite{Ru2} for some results on dissimilarity vectors of positive-weighted and general-weighted trees. In the end, in the recent paper \cite{B-R}, the authors have obtained some theorems about $(n-1)$-dissimilarity vectors of positive-weighted graphs and trees. \\

In this paper we characterize the families of positive real numbers $\{D_{i,j}\}_{i,j \in \{1,...,n\}}$ realized by a positive-weighted tree with exactly $n$ vertices, where $n \in \N, n \geq 3$ (see Theorem \ref{thm:ThmAgne}). 

\section{The main result}

In Theorem \ref{thm:ThmAgne} we will use the following notation:

\begin{nota} \label{nota}

$\bullet$  For any $n \in \N $ with $ n \geq 1$, let $[n]= \{1,..., n\}$.\\

$\bullet$  For any $i$ and $j$ vertices of a graph $G$, we denote with $e(i,j)$ the edge joining $i$ with $j$. \\

$\bullet$  Following \cite{B-M}, we will say that a graph is {\bf simple} if it has no loops or parallel edges; we define {\bf path} every simple graph whose vertices can be arranged in a linear sequence in such a way that two vertices are adjacent if and only if they are consecutive in the sequence.\\

$\bullet$  For simplicity, the vertices of trees will be often named with natural numbers. 
\end{nota}

%%%%%%%%%%%%%%%%%%%%%%%%%%%%%%%%% TEOREMA!!! %%%%%%%%%%%%%%%%%%

\begin{thm}\label{thm:ThmAgne}
Let $n \in \N $ with $ n \geq 3$ and let $\{D_{i,j}\}_{i,j \in [n]}$ be a set of positive real numbers. There exists a positive-weighted tree ${\cal T}=(T,w)$ with exactly $n$ vertices, $1,...,n$, such that $D_{i,j}(T)=D_ {i,j}$ for all $i,j \in [n]$ if and only if the four-point condition holds and also:
\begin{itemize}
\item[$i.$] if for some $i,j,k,t \in [n]$ we have $D_{i,j}+D_{k,t}=D_{i,k}+D_{j,t}=D_{i,t}+D_{j,k}$, then there is an element $l$ in $[n]$ such that 
\begin{equation} \label{Bun1}
D_{u,v}= D_{u,l} + D_{v,l} \quad \textrm{for all distinct } u, v \in \{i,j,k,t\};
\end{equation}

\item[$ii.$] if for some $i,j,k,t \in [n]$ we have $D_{i,j}+D_{k,t}<D_{i,k}+D_{j,t}=D_{i,t}+D_{j,k}$, then, for any choice of three elements $u,v,w \in \{i,j,k,t\}$, there exists an element $l$ in $[n]$ such that 
\begin{equation} \label{Bun2}
D_{u,v}= D_{u,l} + D_{v,l}, \quad D_{u,w}= D_{u,l} + D_{w,l}, \quad D_{v,w}= D_{v,l} + D_{w,l}
\end{equation}
and
\begin{equation} \label{Bun3}
D_{u,v} + D_{w,l}= D_{u,w} + D_{v,l}= D_{u,l} + D_{v,w}.
\end{equation} 
\end{itemize}
Finally, if ${\cal T}$ exists, then it is unique.
\end{thm}

\begin{proof} 
$\Rightarrow$
It is obvious that the the four-point condition is still necessary. Let $\cal{T}=(T,w)$ be a positive-weighted tree with $V(T)=[n]$, and let $i,j,k,t$ be four elements in $V(T)$ such that  
$D_{i,j}(\cal{T})+D_{k,t}(\cal{T})=D_{i,k}(\cal{T})+D_{j,t}(\cal{T})=D_{i,t}(\cal{T})+D_{j,k}(\cal{T})$: 
in this case the minimal subtree of $\cal{T}$ containing the four vertices must be a star, and, if we call $l$ the vertex in the centre of the star, then (\ref{Bun1}) holds. \\ 
If there are $i,j,k,t \in V(T)$ such that 
$D_{i,j}(\cal{T})+D_{k,t}(\cal{T})<D_{i,k}(\cal{T})+D_{j,t}(\cal{T})=D_{i,t}(\cal{T})+D_{j,k}(\cal{T})$, 
then the minimal subtree containing the four vertices must be as in Figure \ref{fig} (if $i,j,k,t$ are distinct); note that $i$ or $j$ could coincide with $l$, and $k$ or $t$ could coincide with $s$. 

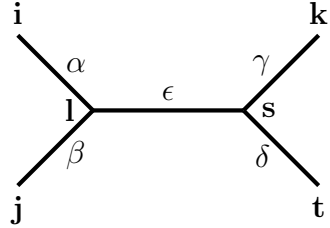
\begin{figure}[h!]
\begin{center}

\begin{tikzpicture}
\draw [ultra thick] (0,0) --(1,1) ; 
\draw [ultra thick] (1,1) --(0,2) ;
\draw [ultra thick] (1,1) --(3,1) ;  
\draw [ultra thick] (3,1) --(4,2) ;
\draw [ultra thick] (3,1) --(4,0) ;  

\node [above] at (0,2) {$\mathbf{i}$};
\node [below] at (0,0) {$\mathbf{j}$};
\node [below] at (4,0) {$\mathbf{t}$};
\node [above] at (4,2) {$\mathbf{k}$};
\node [left] at (0.9,1) {$\mathbf{l}$};
\node [right] at (3.1,1) {$\mathbf{s}$};
\node [right] at (0.5,1.6) {$\mathbf{\alpha}$};
\node [right] at (0.5,0.4) {$\mathbf{\beta}$};
\node [left] at (3.5,0.4) {$\mathbf{\delta}$};
\node [left] at (3.5,1.6) {$\mathbf{\gamma}$};
\node [above] at (2,1) {$\mathbf{\epsilon}$};

\end{tikzpicture}

\caption{Minimal subtree if $D_{i,j}(\cal{T})+D_{k,t}(\cal{T})<D_{i,k}(\cal{T})+D_{j,t}(\cal{T})=D_{i,t}(\cal{T})+D_{j,k}(\cal{T}).$ \label{fig}}
\end{center}
\end{figure}

In this case, if we choose for example $u=i,v=j$ and $w=k$ in $\{i,j,k,t\}$, we can consider the vertex $l$ as in figure and then the equalities (\ref{Bun2}) and (\ref{Bun3}) hold. We argue analogously if $i,j,k$ and $t$ are not distinct. \\

$\Leftarrow$
We proceed by induction on $n$. If $n=3$, then, up to interchanging $1,2$ and $3$, we can suppose that $D_{1,2}=D_{1,3}+D_{2,3}$. In fact, if $D_{1,2}<D_{1,3}+D_{2,3}$, then by condition $(ii)$ there exists $l \in [3]$ such that (\ref{Bun2}) holds. Moreover, $l$ can't be equal to $3$, otherwise we would have $D_{1,2}=D_{1,3}+D_{2,3}$; if $l=1$, then $D_{2,3}=D_{1,2}+D_{1,3}$ and we get an equality, if $l=2$, we get again an equality because $D_{1,3}=D_{1,2}+D_{2,3}$. So up to interchanging $1,2$ and $3$, we can suppose that $D_{1,2}=D_{1,3}+D_{2,3}$. Let $\cal{T}$ be the positive-weighted path with three vertices in Figure \ref{fig2}: obviously $D_{i,j}(\cal{T})=D_{i,j}$ for all $i,j \in [3]$ and $\cal{T}$ is unique.

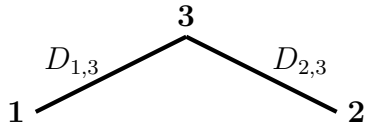
\begin{figure}[h!]
\begin{center}

\begin{tikzpicture}
\draw [ultra thick] (3,1) --(5,0) ;
\draw [ultra thick] (1,0) --(3,1) ;  

\node [right] at (5,0) {$\mathbf{2}$};
\node [above] at (3,1) {$\mathbf{3}$};
\node [left] at (1,0) {$\mathbf{1}$};
\node [right] at (4,0.65) {$D_{2,3}$};
\node [left] at (2,0.65) {$D_{1,3}$};

\end{tikzpicture}

\caption{Positive-weighted tree when $n=3$ and $D_{1,2}=D_{1,3}+D_{2,3}$. \label{fig2}}
\end{center}
\end{figure}

If $n=4$, we consider two cases: 
\begin{itemize}
\item[-] if $D_{1,2}+D_{3,4}=D_{1,3}+D_{2,4}=D_{1,4}+D_{2,3}$, then by $(i)$ there exists $l$ in $[4]$ for which (\ref{Bun1}) is true. Let $\cal{T}$ be the positive-weighted star with center $l$, set of leaves $[4]-\{l\}$ and $w(e(i,l))=D_{i,l}$ for any $i \in [4]-\{l\}$ (see Figure \ref{fig3}.$(a)$). We can easily prove that $D_{i,j}(\cal{T})=D_{i,j}$ for any $i,j \in [4]$. Note that $l$ is unique: otherwise, if for example (\ref{Bun1}) held both with l=1 and with l=2, 
we would have: $D_{2,3}=D_{1,2}+D_{1,3}$ and $D_{1,3}=D_{1,2}+D_{2,3}$, thus $2D_{1,2}=0$, which is absurd. This means that also $\cal{T}$ is unique.

\item[-] if one of the equalities above is an inequality, 
for example $D_{1,2}+D_{3,4} < D_{1,3}+D_{2,4}=D_{1,4}+D_{2,3}$, 
then by condition $(ii)$, chosen the triplet $\{1,2,3\}$ as $\{u,v,w\}$,
there exists $l \in [4] $ such that equalities (\ref{Bun2}) hold. 
Note that $l$ must be different from $4$, otherwise we would have: 
$D_{1,2}=D_{1,4}+D_{2,4}$, $D_{2,3}=D_{2,4}+D_{3,4}$ 
and then $D_{1,2}+D_{3,4} =D_{1,4}+D_{2,3}$, which is impossible. 
Moreover $l$ is different from $3$, because, if $D_{1,2}=D_{1,3}+D_{2,3}$,
then $D_{1,3}+D_{2,3}+D_{3,4} < D_{1,3}+D_{2,4}$, that is $D_{2,3}+D_{3,4} < D_{2,4}$, 
which is absurd by the four-point condition. \\
If $l=2$, let ${\cal T}$ be the  positive-weighted path in Figure \ref{fig3}.$(b)$ if, 
for the triplet $\{1,3,4\}$, $l$ is equal to $3$; 
let ${\cal T}$ be the  positive-weighted path in Figure \ref{fig3}.$(c)$ if,
for the triplet $\{1,3,4\}$, $l$ is equal to $4$. 
It is easy to check that, in both cases, ${\cal T}$ is the unique positive-weighted tree 
with $[4]$ as the set of vertices such that $D_{i,j} ({\cal T})=D_{i,j}$ for all $i,j \in [4]$. \\
Finally, if for the triplet $\{1,2,3\}$ $l$ is equal to $1$, 
let ${\cal T}$ be the  positive-weighted path in Figure \ref{fig3}.$(d)$ if,
for the triplet $\{1,3,4\}$, $l$ is equal to $3$; 
let ${\cal T}$ be the  positive-weighted path in Figure \ref{fig3}.$(e)$ if, 
for the triplet $\{1,3,4\}$, $l$ is equal to $4$. 
It is easy to check that, in both cases, ${\cal T}$ is the unique positive-weighted tree 
with $[4]$ as the set of vertices such that $D_{i,j} ({\cal T})=D_{i,j}$ for all $i,j \in [4]$.

\end{itemize}

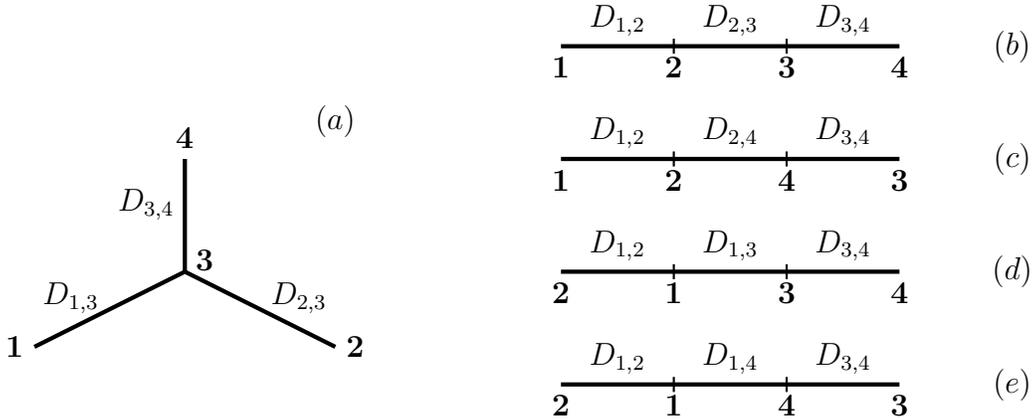
\begin{figure}[h!]
\begin{center}

\begin{tikzpicture}
\draw [ultra thick] (3,1) --(5,0) ;
\draw [ultra thick] (1,0) --(3,1) ;  
\draw [ultra thick] (3,1) --(3,2.5) ;

\node [right] at (5,0) {$\mathbf{2}$};
\node [right] at (3,1.15) {$\mathbf{3}$};
\node [left] at (1,0) {$\mathbf{1}$};
\node [above] at (3,2.5) {$\mathbf{4}$};

\node [right] at (4,0.65) {$D_{2,3}$};
\node [left] at (2,0.65) {$D_{1,3}$};
\node [left] at (3,1.9) {$D_{3,4}$};

\draw [ultra thick] (8,4) --(12.5,4) ;
\draw [ultra thick] (8,2.5) --(12.5,2.5) ;  
\draw [ultra thick] (8,1) --(12.5,1) ;
\draw [ultra thick] (8,-0.5) --(12.5,-0.5) ;

\draw [thick] (9.5,3.9) --(9.5,4.1) ;
\draw [thick] (11,3.9) --(11,4.1) ;  
\draw [thick] (9.5,2.4) --(9.5,2.6) ;
\draw [thick] (11,2.4) --(11,2.6) ;  
\draw [thick] (9.5,0.9) --(9.5,1.1) ;
\draw [thick] (11,0.9) --(11,1.1) ;  
\draw [thick] (9.5,-0.4) --(9.5,-0.6) ;
\draw [thick] (11,-0.4) --(11,-0.6) ;

\node [below] at (8,4) {$\mathbf{1}$};
\node [below] at (9.5,4) {$\mathbf{2}$};
\node [below] at (11,4) {$\mathbf{3}$};
\node [below] at (12.5,4) {$\mathbf{4}$};

\node [below] at (8,2.5) {$\mathbf{1}$};
\node [below] at (9.5,2.5) {$\mathbf{2}$};
\node [below] at (11,2.5) {$\mathbf{4}$};
\node [below] at (12.5,2.5) {$\mathbf{3}$};

\node [below] at (8,1) {$\mathbf{2}$};
\node [below] at (9.5,1) {$\mathbf{1}$};
\node [below] at (11,1) {$\mathbf{3}$};
\node [below] at (12.5,1) {$\mathbf{4}$};

\node [below] at (8,-0.5) {$\mathbf{2}$};
\node [below] at (9.5,-0.5) {$\mathbf{1}$};
\node [below] at (11,-0.5) {$\mathbf{4}$};
\node [below] at (12.5,-0.5) {$\mathbf{3}$};

\node [above] at (8.75,4) {$D_{1,2}$};
\node [above] at (8.75,2.5) {$D_{1,2}$};
\node [above] at (8.75,1) {$D_{1,2}$};
\node [above] at (8.75,-0.5) {$D_{1,2}$};
\node [above] at (10.25,4) {$D_{2,3}$};
\node [above] at (10.25,2.5) {$D_{2,4}$};
\node [above] at (10.25,1) {$D_{1,3}$};
\node [above] at (10.25,-0.5) {$D_{1,4}$};
\node [above] at (11.75,4) {$D_{3,4}$};
\node [above] at (11.75,2.5) {$D_{3,4}$};
\node [above] at (11.75,1) {$D_{3,4}$};
\node [above] at (11.75,-0.5) {$D_{3,4}$};

\node at (5,3) {$(a)$};
\node at (14,4) {$(b)$};
\node at (14,2.5) {$(c)$};
\node at (14,1) {$(d)$};
\node at (14,-0.5) {$(e)$};

\end{tikzpicture}

\caption{$(a)$ $D_{1,2}+D_{3,4}=D_{1,3}+D_{2,4}=D_{1,4}+D_{2,3}$ and $l=3.$ \newline $(b)$, $(c)$, $(d)$  $D_{1,2}+D_{3,4}<D_{1,3}+D_{2,4}=D_{1,4}+D_{2,3}$.\label{fig3}}
\end{center}
\end{figure}

When $n$ is generic, we consider $a, b, c \in [n]$
 such that the value $$D_{a,c} + D_{b,c} - D_{a,b}$$ is maximum 
and the number 
$$ \alpha := \frac{1}{2} (D_{a,c} + D_{a,b} - D_{b,c})$$ 
is positive (note that, up to swapping $a$ with $b$, we can suppose that $\alpha$ is positive; otherwise we would have: $D_{a,c}+D_{a,b}-D_{b,c}=0$ and $D_{b,c}+D_{a,b}-D_{a,c}=0$, thus $D_{a,c}+D_{a,b}-D_{b,c}+D_{b,c}+D_{a,b}-D_{a,c}=0$, and so $2D_{a,b}=0$, which is absurd).

For all $r$ in $[n]-\{a,b\}$ we have:

$$D_{r,c} + D_{b,c} - D_{r,b} \leq D_{a,c} + D_{b,c} - D_{a,b},$$

hence

$$ D_{r,c} + D_{a,b} \leq D_{a,c} + D_{r,b}.$$

Analogously:

$$D_{a,c} + D_{r,c} - D_{a,r} \leq D_{a,c} + D_{b,c} - D_{a,b},$$

hence 

$$D_{r,c} + D_{a,b} \leq D_{b,c} + D_{a,r}.$$

So, by the four-point condition,
\begin{equation} \label{Bun6}
 D_{a,c} + D_{r,b} = D_{b,c} + D_{a,r}.
\end{equation}

Moreover, for all $r \in [n]-\{a,b\}$ we have:

\begin{equation} \label{Bun5}
2 \alpha=D_{a,c} + D_{a,b} - D_{b,c} =  D_{a,r} + D_{a,b} - D_{r,b},
\end{equation}

where the second equality holds by (\ref{Bun6}).

If we choose $s \in [n]-\{a,b\}$, then
\begin{equation} \label{Bun7}
D_{a,c} + D_{s,b} = D_{b,c} + D_{a,s}
\end{equation}

as before, and combining (\ref{Bun6}) with (\ref{Bun7}) we obtain:

\begin{equation} \label{Bun4}
D_{s,b} + D_{a,r} = D_{r,b} + D_{a,s} \geq D_{a,b} + D_{r,s},
\end{equation}
where the last inequality follows from the four-point condition.

Now, fix $r,s \in [n]-\{a,b\}$; we consider the quadruplet $\{a,b,r,s\}$; there are two possibilities:

\begin{enumerate}
\item Suppose that $D_{a,b}+D_{r,s}=D_{a,r}+D_{b,s}=D_{a,s}+D_{b,r}.$ Then, by $(i)$,
there exists $l \in [n]-\{a\}$ such that (\ref{Bun1}) holds, 
and, using (\ref{Bun5}), such that $\alpha = D_{a,l}$. 
Observe that $l$ is different from $a$, because, otherwise, we would have:
$$D_{r,b}= D_{a,r} + D_{a,b},$$
hence
$$ 0=D_{a,r} + D_{a,b} - D_{r,b}= 2 \alpha$$
(where the last equality holds by (\ref{Bun5})), and this is not possible. 
Let us consider the set $[n]-\{a\};$ 
by inductive hypothesis there exists a unique positive-weighted tree 
$\cal{T'}=(T',w')$, with $V(T')=[n]-\{a\}$, such that 
$D_{x,y}(\cal{T'})=D_{x,y}$ for any $x,y \in [n]-\{a\}$. 
To realize $\cal{T}$ we attach an edge with weight $\alpha$ to the vertex $l$, 
and we call the second vertex $a$. We obtain a weighted tree $\cal{T}$ with exactly $n$ vertices and we have:
\begin{itemize}
\item [-] $D_{a,l}(\cal{T})=\alpha = D_{a,l}$;
\item [-] $D_{a,b}(\cal{T})=\alpha + D_{l,b}(\cal{T'}) =D_{a,l} + D_{l,b} = D_{a,b}$;
\item [-] $D_{x,y}(\cal{T})=D_{x,y}(\cal{T}')=D_{x,y } \quad$  for any $x , y \neq a$;
\item [-] $D_{a,x}(\cal{T})=\alpha + D_{l,x} (\cal{T'})=D_{a,l} + D_{l,x}\quad$ for any $x \in [n]-\{a,b,l\}.$
\end{itemize}

So we have to demonstrate that $D_{a,l} + D_{l,x}=D_{a,x}$ for all $x \in [n]-\{a,b,l\}$; we already know that 
$D_{a,l} + D_{l,x} \geq D_{a,x}$, so it is enough to prove $D_{a,l} + D_{l,x} \leq D_{a,x}.$

If $l$ is different from $b$, then, by (\ref{Bun4}), with $l$ and $x$ instead of $r$ and $s$, we have:
$$D_{a,b} + D_{l,x} \leq D_{b,x} + D_{a,l} = D_{l,b} + D_{a,x},$$
hence
$$D_{a,x} \geq D_{a,b} + D_{l,x} - D_{l,b} = D_{a,l} + D_{l,x},$$
where the last equality follows from (\ref{Bun1}).

If $l$ is equal to $b$, then we use again (\ref{Bun4}), with $x$ instead of $s$, and we obtain:
$$D_{a,x}=D_{a,r} + D_{b,x} - D_{r,b}= D_{a,b} + D_{r,b}+D_{b,x} - D_{r,b}=D_{a,b} + D_{b,x}.$$

To prove that $\cal{T}$ is unique, 
suppose that there exists a positive-weighted tree, $\cal{A}=(A,w_A)$, 
different from $\cal{T}$, such that $V(A)=[n]$ and $D_{x,y}(\cal{A})=D_{x,y}$ for all $x,y \in [n].$
First note that $a$ must be a leaf in both trees, for example let us prove that $a$ must be a leaf in $\cal{T}:$
otherwise there would exist elements $x$ and $y$ in $[n]-\{a,b\}$ with 
$D_{x,y}(\cal{T})=D_{x,a}(\cal{T})+D_{a,y}(\cal{T})$, that is 
$D_{x,y}=D_{x,a}+D_{a,y}$. By (\ref{Bun4}), with $x$ and $y$ instead of $r$ and $s$, we know that:

$$D_{x,b} \geq  D_{x,y} + D_{a,b} - D_{y,a} =D_{x,a}+D_{a,b},$$

that implies by (\ref{Bun5}):

$$2\alpha=D_{x,a}+D_{a,b}-D_{x,b}\leq 0, $$

which is impossible. Analogously we can prove that $a$ is a leaf in $\cal{A}$.\\
Let us consider the positive-weighted tree $\cal{A'}=(A',w'_A)$ 
obtained from $\cal{A}$ by deleting the leaf $a$ with its twig: 
$V(A')=[n]-\{a\}$ and $D_{x,y}(\cal{A})=D_{x,y}$ for all $x,y \in [n]- \{a\},$ 
so by inductive hypothesis $\cal{A'}$ must be equal to $\cal{T'}$.

To end the proof we observe that the unique way to reconstruct the tree $\cal{A}$ 
from $\cal{A'}$ (and analogously the unique way to reconstruct the tree $\cal{T}$ 
from $\cal{T'}$) is attaching an edge with weight $\alpha$ and with leaf $a$ 
to the unique vertex $l$ adjacent to $a$ in $\cal{A}$. For this vertex we have that:

$$D_{a,x}(\cal{A})=D_{a,l}(\cal{A})+D_{l,x}(\cal{A})$$ 

for any $x \in [n]-\{a,l\}$, hence:

\begin{equation} \label{BunFin}
D_{a,x}=D_{a,l}+D_{l,x}.
\end{equation}

Observe that, if there existed $l$ and $\tilde{l}$ in $[n]$ such that (\ref{BunFin}) holds, then we would have:
$$D_{a,\tilde{l}}=D_{a,l}+D_{l,\tilde{l}}\,\, \textrm{ and } \,\, D_{a,l}=D_{a,\tilde{l}}+D_{l,\tilde{l}}$$
and so $2D_{l,\tilde{l}}=0,$ which implies $l=\tilde{l}$. Hence $\cal{T}$ and $\cal{A}$ are obtained from the same tree by attaching an edge to the same vertex, so they are equal.

\item Suppose that $D_{a,b}+D_{r,s}<D_{a,r}+D_{b,s}=D_{a,s}+D_{b,r}.$ 
Then by (\ref{Bun2}), choosen the triplet $\{a,b,r\}$ as $\{u,v,w\}$, there exists $\tilde{l} \in [n]$ such that:

\begin{equation} \label{lastBun}
D_{a,b}+D_{r,\tilde{l}}=D_{a,r}+D_{b,\tilde{l}}=D_{a,\tilde{l}}+D_{b,r},
\end{equation}

$$D_{a,b}=D_{a,\tilde{l}}+D_{b,\tilde{l}}, \quad D_{a,r}=D_{a,\tilde{l}}+D_{r,\tilde{l}}, \quad D_{b,r}=D_{b,\tilde{l}}+D_{r,\tilde{l}}.$$
We have $\tilde{l}\neq a$, so, if it is different from $b$, considering the quadruplet $\{a, b ,l ,r \}$, we can return to the first case. \\
If $\tilde{l}=b,$ we consider the unique positive-weighted tree $\cal{T'}=(T',w')$ such that $V(T')=[n]-\{a\}$ and 
$D_{x,y}(\cal{T'})=D_{x,y}$ for any $x,y \in [n]-\{a\}$. 
Then we construct the tree $\cal{T}$ by attaching an edge with weight $D_{a,b}$ to the vertex $b$, 
and by calling its second vertex $a$. In this tree we have:
\begin{itemize}
\item $D_{a,b}(\cal{T})=D_{a,b};$
\item $D_{x,y}(\cal{T})=D_{x,y}(\cal{T'})=D_{x,y} \quad$ for all $x,y \in [n]-\{a\};$ 
\item $D_{a,x}(\cal{T})=D_{a,b}(\cal{T})+D_{x,b}(\cal{T'})=D_{a,b}+D_{x,b} \quad$ for all $x \in [n]-\{a,b\}.$ 
\end{itemize}
So we have to demonstrate that $D_{a,x}=D_{a,b}+D_{b,x}$ for all $x \in [n]-\{a,b\}$. By (\ref{Bun4}), with $x$ instead of $s$, we know that: $$D_{a,x}=D_{b,x}+D_{r,a}-D_{r,b},$$
and by (\ref{lastBun}) we know that: 
$$D_{a,r}=D_{a,b}+D_{r,b},$$ 
so we can conclude.\\
The uniqueness of $\cal{T}$ can be proved as in the previous case. 
\end{enumerate}
\end{proof}

\bigskip

{\bf Address of the author:}
Dipartimento di Matematica e Informatica ``U. Dini'', 
viale Morgagni 67/A,
50134  Firenze, Italia 

{\bf
E-mail addresse:}
baldisser@math.unifi.it

\end{document}